\documentclass[11pt]{article}
\usepackage[utf8]{inputenc}
\usepackage[a4paper,width=145mm,top=20mm,bottom=20mm,bindingoffset=6mm]{geometry}
\usepackage{etex}
\usepackage{etoolbox} 
\usepackage{lmodern}
\usepackage{array}
\usepackage{authblk}
\usepackage{xstring}
\usepackage{longtable}
\usepackage[listings]{tcolorbox} 
\tcbuselibrary{breakable}
\tcbuselibrary{skins}
\usepackage[pdftex]{hyperref}
\usepackage{amsmath}
\usepackage{tikz,tikz-cd,stmaryrd}
\usepackage{fancyhdr}
\usepackage{amssymb}
\usepackage{amsthm}
\usepackage{amsmath} 
\usepackage{amsfonts}
\usepackage{amsthm}
\usepackage{pifont}
\usepackage{amsmath}
\usepackage{epsfig}
\usepackage{latexsym}
\usepackage{amssymb}
\usepackage{color}
\usepackage{amscd}
\usepackage{multirow}
\usepackage{graphicx}
\usepackage{lscape}
\usepackage{float}
\usepackage{xcolor}
\usepackage{bm}
\usepackage{tikz-cd}
\usepackage{authblk}
\newcommand{\ad}{\operatorname{ad}}
 \makeatletter
    \let\stdchapter\section
    \renewcommand*\section{%
    \@ifstar{\starchapter}{\@dblarg\nostarchapter}}
    \newcommand*\starchapter[1]{%
        \stdchapter*{#1}
        \thispagestyle{fancy}
        \markboth{\MakeUppercase{#1}}{}
    }
    \def\nostarchapter[#1]#2{%
        \stdchapter[{#1}]{#2}
        \thispagestyle{fancy}
    }
\makeatother
 
\newtheorem{theorem}{Theorem}[section]
\newtheorem*{theorem*}{Theorem}
\newtheorem{lemma}[theorem]{Lemma}
\newtheorem{proposition}[theorem]{Proposition}
\theoremstyle{definition}
\newtheorem{definition}[theorem]{Definition}
\theoremstyle{corollary}
\newtheorem{corollary}[theorem]{Corollary}
\theoremstyle{remark}

\newtheorem{remark}[theorem]{Remark}
 
\begin{document}
\title{Complex structures on stratified Lie algebras}

\author{Junze Zhang%
  \thanks{Electronic address: \texttt{z5055984@ad.unsw.edu.au}}}
\affil{School of Mathematics and Statistics, University of New South Wales, Sydney}

\date{\today}
\maketitle
 
\pagestyle{fancy}
\fancyhead{}
\fancyfoot{}
\fancyfoot[LE,RO]{\thepage}
\fancyhead[L]{\slshape\MakeUppercase{Complex structures on stratified Lie algebras}}
\newcommand{\underuparrow}[1]{\ensuremath{\underset{\uparrow}{#1}}}

\begin{abstract}
This paper investigates some properties of complex structures on Lie algebras. In particular, we focus on $\textit{nilpotent}$ $\textit{complex structures}$ that are characterized by a suitable $J$-invariant ascending or descending central series $\mathfrak{d}^j$ and $\mathfrak{d}_j$ respectively. In this article, we introduce a new descending series $\mathfrak{p}_j$ and use it to give proof of a new characterization of nilpotent complex structures. We examine also whether nilpotent complex structures on stratified Lie algebras preserve the strata. We find that there exists a $J$-invariant stratification on a step $2$ nilpotent Lie algebra with a complex structure. 
\end{abstract}

\section{Introduction}
 \label{a}

In recent years, complex structures on nilpotent Lie algebras have been shown to be very useful for understanding some geometric and algebraic properties of nilmanifolds. In $\cite{MR1665327}$ and $\cite{MR1899353}$, Cordero, Fern{\'a}ndez, Gray and Ugarte introduced $\textit{nilpotent complex}$ $\textit{structures} $, studied $6$ dimensional nilpotent Lie algebras with nilpotent complex structures, and provided a classification. Since the ascending central series is not necessarily $J$-invariant, they introduced a $J$-invariant ascending central series to characterize nilpotent complex structures. More recently, Latorre, Ugarte and Villacampa defined the space of nilpotent complex structures on nilpotent Lie algebras and further studied complex structures on nilpotent Lie algebras with one dimensional center $\cite{MR4009385}$, $ \cite{latorre2020complex}$. They also provided a theorem describing the ascending central series of $8$ dimensional nilpotent Lie algebras with complex structures. In $\cite{gao2020maximal}$, Gao, Zhao and Zheng studied the relation between the step of a nilpotent Lie algebra and the smallest integer $j_0$ such that the $J$-invariant ascending central series stops. Furthermore, they introduced a $J$-invariant descending central series, which is another tool to characterize nilpotent complex structures. These papers use the language of differential forms to characterize nilpotent complex structures. Our proofs here in this paper are purely Lie algebraic. 
 
Let $G$ be a Lie group and $\mathfrak{g} \cong T_eG$ be its Lie algebra, which we always assume to be real, unless otherwise stated. A linear isomorphism $J:TG \rightarrow TG$ is an $\textit{almost complex structure}$ if $J^2 = -I.$ By the Newlander--Nirenberg Theorem $\cite{MR88770},$ an almost complex structure $J$ corresponds to a left invariant complex structure on $G$ if and only if   \begin{equation}
    [J_eX,J_eY] -[X,Y] - J_e([J_eX,Y]+[X,J_eY]) = 0,   \label{eq:Niv}
\end{equation} for all $X,Y \in \mathfrak{g}.$ Since we are interested only on Lie algebras in this paper, from now on, we will write $J$ for $J_e.$ We will refer to $\eqref{eq:Niv}$ as the $\textit{Newlander--Nirenberg condition}.$

\section{Complex structures on nilpotent Lie algebras}
\label{b}

In this section, we consider some properties of the central series of nilpotent Lie algebras with complex structures $J$ and define $J$-invariant central series. We define nilpotent complex structures, and relate their properties to the dimension of the center $\mathfrak{z}$ of a nilpotent Lie algebra.
 
\begin{definition}
\label{acde} ($\text{See}, \text{ for instance}$, e.g., $\cite{MR1920389}$)
Let $\mathfrak{g}$ be a Lie algebra. The $\textit{descending}$ $\textit{central series}$ and $\textit{ascending}$ $\textit{central series}$ of $\mathfrak{g} $ are denoted $\mathfrak{c}_j(\mathfrak{g})$ and $\mathfrak{c}^j(\mathfrak{g})$ respectively, for all $j \geq 0 ,$ and defined inductively by \begin{eqnarray}
&  \mathfrak{c}_0(\mathfrak{g}) = \mathfrak{g},  \text{ }  \mathfrak{c}_j(\mathfrak{g}) = [\mathfrak{g},\mathfrak{c}_{j-1}(\mathfrak{g})] \label{eq:de}   ; \\  
& \mathfrak{c}^0(\mathfrak{g}) = \{0\},  \text{ } \text{ }    \mathfrak{c}^j(\mathfrak{g}) = \{X \in \mathfrak{g}:[X,\mathfrak{g}] \subseteq \mathfrak{c}^{j-1}(\mathfrak{g})\}.   \label{eq:ac}
\end{eqnarray}
\end{definition}

\begin{remark}
\label{racde}
(i) Notice that $\mathfrak{c}^1(\mathfrak{g})=\mathfrak{Z}(\mathfrak{g})  $, $\mathfrak{c}_1(\mathfrak{g}) = [\mathfrak{g},\mathfrak{g}],$ and \begin{align*}
    \mathfrak{c}^{j}(\mathfrak{g})/\mathfrak{c}^{j-1}(\mathfrak{g}) = \mathfrak{Z}\left(\mathfrak{g}/ \mathfrak{c}^{j-1}(\mathfrak{g})\right) \text{ }  \text{for all } j \geq 1,
\end{align*} where $\mathfrak{Z}(\cdot)$ means the center of a Lie algebra. Furthermore, $\mathfrak{c}_j(\mathfrak{g})/\mathfrak{c}_{j+1}(\mathfrak{g})$ $\subseteq$  $\mathfrak{Z} \left(\mathfrak{n}/\right. $   $\left.\mathfrak{c}_{j+1}(\mathfrak{g})\right)$ for all $j \geq 0.$ It is clear that $\mathfrak{c}^j(\mathfrak{g})$ and $\mathfrak{c}_j(\mathfrak{g})$ are ideals of $\mathfrak{g}$ for all $j \geq 0$.
 
(ii) A Lie algebra $\mathfrak{g}$ is called $\textit{nilpotent of step }k $, for some $k \in \mathbb{N},$ if $\mathfrak{c}_k(\mathfrak{g}) = \{0\}$ and $\mathfrak{c}_{k-1}(\mathfrak{g}) \neq \{0\}.$ We will denote nilpotent Lie algebras by $\mathfrak{n}$ in this paper. See, e.g., $\cite[\text{Section }5.2]{MR3025417}$ or $\cite{MR1920389}.$  
\end{remark}

\subsection{$J$-invariant central series and nilpotent complex structures} 
 \label{b4}

Following $\cite[\text{Definition }1]{MR1665327}$, we define the $\textit{J-invariant}$ $\textit{ascending}$ $\textit{central series}$ $\mathfrak{d}^j$ for nilpotent Lie algebras and introduce $\textit{nilpotent}$ $\textit{complex}$ $\textit{structures}$ on nilpotent Lie algebras. Furthermore, we recall the definition of the $\textit{J-invariant}$ $\textit{descending}$ $\textit{central series}$ $\mathfrak{d}_j$. $\cite[\text{Defintion 2.7}]{gao2020maximal}$ 

\begin{definition}
\label{7}
Let $ \mathfrak{n} $ be a Lie algebra with a complex structure $J.$ Define a sequence of $J$-invariant ideals of $\mathfrak{n}$ by $\mathfrak{d}^0 = \{0\}$ and \begin{equation}
      \mathfrak{d}^j  = \{ X \in \mathfrak{n} : [X,\mathfrak{n}] \subseteq  \mathfrak{d}^{j-1} , [JX,\mathfrak{n}] \subseteq \mathfrak{d}^{j-1}  \}   \label{eq:1} 
\end{equation} for all $ j \geq 1.$ We call the sequence $\mathfrak{d}^j$ the $\textit{ascending J-invariant}$ $\textit{central series}.$ The complex structure $J$ is called $\textit{nilpotent of step }j_0$ if there exists $j_0 \in \mathbb{N}$ such that $\mathfrak{d}^{j_0} = \mathfrak{n}$ and $\mathfrak{d}^{j_0-1} \subset \mathfrak{n}$. 

We define inductively the $J$-$\textit{invariant}$ $ \textit{descending}$ $ \textit{central}$  $\textit{series}$ by    \begin{equation}
    \mathfrak{d}_0 = \mathfrak{n}, \text{ } \text{ } \mathfrak{d}_j = [\mathfrak{d}_{j-1},\mathfrak{n}] + J[\mathfrak{d}_{j-1},\mathfrak{n}]  \text{ } \text{ }  \label{eq:00}
\end{equation} all for $j \geq 1.$ 
\end{definition}

\begin{remark}
\label{r7}
(i) For the ascending $J$-invariant central series $\mathfrak{d}^j$, \begin{align*}
    \mathfrak{d}^j / \mathfrak{d}^{j-1} = \mathfrak{Z}(\mathfrak{n}/\mathfrak{d}^{j-1}) \cap J\mathfrak{Z}(\mathfrak{n}/\mathfrak{d}^{j-1}) \qquad   \text{ for all } j \geq 1 .
\end{align*}  In particular, $\mathfrak{d}^1 = \mathfrak{z} \cap J\mathfrak{z},$ which is the largest $J$-invariant subspace of $\mathfrak{z} $ and,  if $J$ is nilpotent, then $\mathfrak{d}^1 \neq \{0\}$. The nilpotency of $J$ implies that the ascending $J$-invariant central series $\mathfrak{d}^j$ of $\mathfrak{n}$ is strictly increasing until $\mathfrak{d}^{j_0} = \mathfrak{n}$. Furthermore, if $\mathfrak{n}$ is a step $k$ nilpotent Lie algebra with a nilpotent complex structure $J$ of step $j_0$, then $k \leq j_0 \leq \frac{1}{2}\dim \mathfrak{n} .$ See, e.g., $\cite{MR1665327}$ and $\cite{gao2020maximal}.$
 
(ii) By definition, if $\mathfrak{n}$ admits a nilpotent complex structure, then $\mathfrak{n}$ is nilpotent.  

(iii) For all $j \geq 0,$ it is clear that $ \mathfrak{c}_j(\mathfrak{n}) + J  \mathfrak{c}_j(\mathfrak{n}) \subseteq \mathfrak{d}_j$;   Furthermore, $\mathfrak{d}_j \unlhd \mathfrak{n}$ and $\mathfrak{d}^j \unlhd \mathfrak{n}$ where $\unlhd$ is the notation of ideal.  

(iv) Let $ \mathfrak{n} $ be a Lie algebra with a complex structure $J.$  Then $J$ preserves all terms of $ \mathfrak{c}^j(\mathfrak{n})$ if and only if $\mathfrak{d}^j =  \mathfrak{c}^j(\mathfrak{n})$ for all $j \geq 0$. $\cite[\text{Corollary }5]{MR1665327} $ Similarly, $J$ preserves all terms of $\mathfrak{c}_j(\mathfrak{n})$ if and only if $\mathfrak{d}_j = \mathfrak{c}_j(\mathfrak{n})$ for all $j.$ 
 
\end{remark}
 
The following lemma provides a connection between $J$-invariant ascending and descending central series.

\begin{lemma}
\label{9}
Let $\mathfrak{n}$ be a Lie algebra with a complex structure $J$. Suppose that $J$ is nilpotent of step $j_0.$

(i) Then $\mathfrak{n}/\mathfrak{d}^{j_0-1}$ is Abelian. Conversely, if there exists $j_0\in \mathbb{N}$ such that $\mathfrak{n}/\mathfrak{d}^{j_0-1}$ is Abelian, then $J$ is nilpotent of step at most $j_0.$

(ii) Then $\mathfrak{d}_j \subseteq \mathfrak{d}^{j_0 - j} $ for all $j \geq 0 $. Conversely, if there exists $j_0 \in \mathbb{N}$ such that $\mathfrak{d}_j \subseteq \mathfrak{d}^{j_0 - j} $ for all $j \geq 0,$ then $J$ is nilpotent of step at most $j_0.$    
\end{lemma}
 
\begin{proof}
For part (i), suppose that $J$ is nilpotent of step $j_0$. By definition, $\mathfrak{d}^{j_0} = \mathfrak{n}$ and $\mathfrak{d}^{j_0-1} \subset \mathfrak{n}$. Then  \begin{align*}
    \mathfrak{Z}(\mathfrak{n}/\mathfrak{d}^{j_0-1}) \cap J\mathfrak{Z}(\mathfrak{n}/ \mathfrak{d}^{j_0-1}) = \mathfrak{n}/\mathfrak{d}^{j_0-1}.
\end{align*} It is obvious that $ \mathfrak{Z}(\mathfrak{n}/\mathfrak{d}^{j_0-1}) = \mathfrak{n}/\mathfrak{d}^{j_0-1}.$ Hence $\mathfrak{n}/\mathfrak{d}^{j_0-1}$ is Abelian.

Conversely, suppose that there exists $j_0 \in \mathbb{N}$ such that $ \mathfrak{n}/ \mathfrak{d}^{j_0-1}$ is Abelian. Then  $ \{0\} \neq \mathfrak{c}_1(\mathfrak{n}) \subseteq \mathfrak{d}^{j_0-1}$. For all $X  \in \mathfrak{n},$  \begin{align*}
    [X,\mathfrak{n}] \subseteq  \mathfrak{d}^{j_0 - 1} \text{ and }  [JX,\mathfrak{n}] \subseteq \mathfrak{d}^{j_0 -1}.
\end{align*} We deduce that $ \mathfrak{n} = \mathfrak{d}^{j_0}$ and therefore $J$ is nilpotent of step at most $j_0$.  

For part (ii), assume that $J$ is nilpotent of step $j_0.$ By definition, $\mathfrak{d}_0 = \mathfrak{n} = \mathfrak{d}^{j_0}.$ Next, assume that $\mathfrak{d}_{s-1} \subseteq \mathfrak{d}^{j_0-s+1}$ for some $s \in \mathbb{N}.$ Then \begin{align*}
    \mathfrak{d}_s & = [\mathfrak{d}_{s-1},\mathfrak{n}] + J[\mathfrak{d}_{s-1},\mathfrak{n}] \\
    & \subseteq [\mathfrak{d}^{j_0 -s + 1},\mathfrak{n}] + J [\mathfrak{d}^{j_0 -s + 1},\mathfrak{n}] \\
    & \subseteq \mathfrak{d}^{j_0-s} + J\mathfrak{d}^{j_0-s}  = \mathfrak{d}^{j_0-s}.
\end{align*} Hence by induction, $\mathfrak{d}_j \subseteq \mathfrak{d}^{j_0 - j} $ for all $j \geq 0.$ 

Conversely, suppose that there exists $j_0 \in \mathbb{N}$ such that $\mathfrak{d}_j \subseteq \mathfrak{d}^{j_0 - j} $ for all $j \geq 0$. In particular, $  \mathfrak{d}_1 \subseteq \mathfrak{d}^{j_0 -1}.$ By definition, $\mathfrak{c}_1(\mathfrak{n}) \subseteq \mathfrak{d}_1.$ It follows that $$   [\mathfrak{n}/\mathfrak{d}^{j_0-1}, \mathfrak{n}/\mathfrak{d}^{j_0-1}] \subseteq [\mathfrak{n},\mathfrak{n}] + \mathfrak{d}^{j_0-1} = \mathfrak{c}_1(\mathfrak{n}) + \mathfrak{d}^{j_0-1} \subseteq \mathfrak{d}_1 + \mathfrak{d}^{j_0-1} \subseteq \mathfrak{d}_{j_0-1},$$ and thus $\mathfrak{n}/\mathfrak{d}^{j_0 -1}$ is Abelian. From $\text{Lemma } \ref{9}$, $J$ is nilpotent of step at most $j_0.$    
\end{proof}
  
\begin{remark}
\label{r41}
Under the condition of Lemma $\ref{9},$ if $J$ is nilpotent of step $j_0,$ $\mathfrak{d}_{j_0-1} \subseteq \mathfrak{d}^1 \subseteq \mathfrak{z}.$ Then $\mathfrak{d}_{j_0-1}$ is Abelian. Furthermore, there exists $j_0 \in \mathbb{N}$ such that $\mathfrak{n}/\mathfrak{d}^{j_0-1}$ is Abelian if and only if $\mathfrak{d}_j \subseteq \mathfrak{d}^{j_0-j}$ for all $j \geq 0.$ This is proved by induction as in the proof of $\text{Lemma }\ref{9}$.  
\end{remark}

\begin{corollary}
Let $\mathfrak{n}$ be a step $k$ nilpotent Lie algebra with a complex structure $J.$ Then $J$ is nilpotent of step $k$ if and only if $\mathfrak{d}_j \subseteq \mathfrak{d}^{k-j}$ for all $j \geq 0$.
\end{corollary} 
\begin{proof}
Suppose that $J$ is nilpotent of step $k.$ By $\text{Lemma }\ref{9},$ $\mathfrak{d}_j \subseteq \mathfrak{d}^{k-j}$ for all $j \geq 0.$ Conversely, assume that $\mathfrak{d}_j \subseteq \mathfrak{d}^{k-j}$ for all $j.$ Again by Lemma $\ref{9}$, $J$ is nilpotent of step at most $k.$ Furthermore, it follows that $ \{0\} \neq \mathfrak{c}_{k-1}(\mathfrak{n}) \subseteq  \mathfrak{d}_{k-1}.$ Therefore $\mathfrak{d}_{k-1} \neq \{0\} $ and $J$ is nilpotent of step $k.$
\end{proof}

\begin{remark}
From Remark $\ref{r41},$ $J$ is nilpotent of step $k$ if and only if $\mathfrak{n}/\mathfrak{d}^{k-1}$ is Abelian.
\end{remark}
 
We introduce a new descending central series whose descending `rate' is slower than that of $\mathfrak{c}_j(\mathfrak{n})$ but faster than that of $\mathfrak{d}_j.$
 
\begin{definition}
\label{p3}
Let $J$ be a complex structure on a Lie algebra $\mathfrak{n}.$ We define the sequence $\mathfrak{p}_j$ inductively by  \begin{equation}
     \mathfrak{p}_0 = \mathfrak{n} \text{ and }     \mathfrak{p}_j =  [\mathfrak{p}_{j-1},\mathfrak{n}] + [J\mathfrak{p}_{j-1},\mathfrak{n}]  \text{ for all } j \geq 1 . \label{eq:p}
 \end{equation}  
\end{definition}

\begin{remark}
 It is clear that $\mathfrak{p}_{j+1} \subseteq \mathfrak{p}_j$ for all $j \geq 0.$  Furthermore, $\mathfrak{p}_j \unlhd \mathfrak{n}$ since $[\mathfrak{p}_j,\mathfrak{n}] \subseteq \mathfrak{p}_{j+1} \subseteq \mathfrak{p}_j$ for all $j \geq 0$.
\end{remark}
 
\begin{lemma}
\label{p}
Let $\mathfrak{n}$ be a Lie algebra with a complex structure $J $. Then $ \mathfrak{c}_j(\mathfrak{n}) \subseteq \mathfrak{p}_j$ for all $j \geq 0$. Furthermore, $\mathfrak{p}_j \subseteq \mathfrak{d}_j $ and $ J  \mathfrak{p}_j  \subseteq   \mathfrak{d}_j $ for all $j \geq 0$.  
\end{lemma}

\begin{proof}
By definition, $\mathfrak{c}_0(\mathfrak{n}) = \mathfrak{n} = \mathfrak{p}_0.$ It follows, by induction, that $ \mathfrak{c}_j(\mathfrak{n}) \subseteq \mathfrak{p}_j$ for all $j \geq 0$. Using $\eqref{eq:00},$ $[\mathfrak{d}_{j-1},\mathfrak{n}] \subseteq \mathfrak{d}_j $. By definition, $ \mathfrak{p}_0 = \mathfrak{n} = \mathfrak{d}_0$ and $J\mathfrak{p}_0 = J\mathfrak{n} = \mathfrak{n} = \mathfrak{d}_0$. Next, suppose that $\mathfrak{p}_s \subseteq \mathfrak{d}_s$ and $J\mathfrak{p}_s \subseteq \mathfrak{d}_s$ for some $s \in \mathbb{N}.$ Then by $\eqref{eq:p},$ \begin{align*}
     \mathfrak{p}_{s+1}  = [\mathfrak{p}_s,\mathfrak{n}] + [J\mathfrak{p}_s,\mathfrak{n}] \subseteq  [\mathfrak{d}_s,\mathfrak{n}] \subseteq \mathfrak{d}_{s+1} \text{ and } J\mathfrak{p}_{s+1} \subseteq J[\mathfrak{d}_s,\mathfrak{n}] \subseteq \mathfrak{d}_{s+1}
 \end{align*} By induction, $   \mathfrak{p}_j  \subseteq   \mathfrak{d}_j $ and $J\mathfrak{p}_j  \subseteq   \mathfrak{d}_j  $ for all $j \geq 0.$   
\end{proof}

 \begin{remark}
 \label{rp}
(i)  Notice that $\mathfrak{p}_j/\mathfrak{p}_{j+1} \subseteq \mathfrak{Z}\left(\mathfrak{n}/\mathfrak{p}_{j+1}\right)$ for all $j \geq 0.$ Indeed, for all $P \in \mathfrak{p}_j$ and $Y \in \mathfrak{n},$ since $[P,Y] \subseteq \mathfrak{p}_{j+1},$ it is enough to deduce \begin{equation*}
    [P+ \mathfrak{p}_{j+1},Y + \mathfrak{p}_{j+1}] =[P,Y] + \mathfrak{p}_{j+1} \subseteq \mathfrak{p}_{j+1}. 
\end{equation*} 
Hence $\mathfrak{p}_j/\mathfrak{p}_{j+1} \subseteq \mathfrak{Z}\left(\mathfrak{n}/\mathfrak{p}_{j+1}\right)$.  

(ii) By $\text{Lemma }\ref{p},$  $  \mathfrak{p}_j + J \mathfrak{p}_j \subseteq \mathfrak{d}_j$ for all $j \geq 0$. We show that $\mathfrak{p}_j + J \mathfrak{p}_j \unlhd \mathfrak{n}$ for all $j \geq 0.$ Indeed, for all $P,P' \in \mathfrak{p}_j,$ $$
    \underbrace{[P+JP',\mathfrak{n}]}_\text{$\subseteq [\mathfrak{p}_j + J \mathfrak{p}_j,\mathfrak{n}]$} \subseteq  \underbrace{[P,\mathfrak{n}]}_\text{$ \subseteq \mathfrak{p}_{j+1}$} +  \underbrace{[JP',\mathfrak{n}]}_\text{$ \subseteq \mathfrak{p}_{j+1}$} \subseteq \mathfrak{p}_{j+1} \subseteq \mathfrak{p}_{j+1} + J\mathfrak{p}_{j+1}\subseteq \mathfrak{p}_j + J\mathfrak{p}_j.$$ Hence $\mathfrak{p}_j + J \mathfrak{p}_j \unlhd \mathfrak{n}$. From part (ii), we can show that $\mathfrak{p}_j + J\mathfrak{p}_j$ is a $J$-invariant descending central series. Indeed, for all $T = P + JP' \in \mathfrak{p}_j + J \mathfrak{p}_j$ and $Y \in \mathfrak{n},$ \begin{align*}
        [T+ \mathfrak{p}_{j+1} + J\mathfrak{p}_{j+1},Y + \mathfrak{p}_{j+1} +J\mathfrak{p}_{j+1}] \subseteq [T,Y] + \mathfrak{p}_{j+1}  +J\mathfrak{p}_{j+1}\subseteq \mathfrak{p}_{j+1}  +J\mathfrak{p}_{j+1}.
    \end{align*} 
 \end{remark}
 
 \begin{theorem}
\label{14}
Let $\mathfrak{n}$ be a Lie algebra with a complex structure $J.$ The following are equivalent:  
  
(i)   $J$ is nilpotent of step $j_0$;
  
(ii) $ \mathfrak{p}_{j_0} = \{0\}$ and $ \mathfrak{p}_{j_0-1} \neq \{0\};$  
 
(iii) $ \mathfrak{d}_{j_0} = \{0\}$ and $\mathfrak{d}_{j_0 - 1} \neq \{0\}$. 
\end{theorem}

\begin{proof}
We first show that (i) and (ii) are equivalent. Assume that $J$ is nilpotent of step $j_0.$ From $\text{Lemma }\ref{9}$ part (ii), $\mathfrak{d}_{j_0-1} \subseteq \mathfrak{d}^1.$ Hence by $\text{Lemma }\ref{p},$   $$   \mathfrak{p}_{j_0}\subseteq   [\mathfrak{d}_{j_0-1},\mathfrak{n}] \subseteq  [\mathfrak{d}^1,\mathfrak{n}]= \{0\}.$$ Thus $   \mathfrak{p}_{j_0}= \{0\} $. Assume, by contradiction, that $ \mathfrak{p}_{j_0-1} = \{0\}.$ We show that $\mathfrak{p}_{j_0-j-1} + J\mathfrak{p}_{j_0-j-1} \subseteq \mathfrak{d}^j $ for all $j \geq 0 $ by induction. By definition, $  \mathfrak{p}_{j_0-1} + J \mathfrak{p}_{j_0-1} = \{0\} = \mathfrak{d}^0.$ Next, suppose that $ \mathfrak{p}_{j_0-s-1} + J   \mathfrak{p}_{j_0-s-1}\subseteq\mathfrak{d}^s $ for some $s \in \mathbb{N}.$ Then from $\text{Remark }\ref{rp}$ part (ii),  \begin{equation*}
   [\mathfrak{p}_{j_0-s-2} + J \mathfrak{p}_{j_0-s-2},\mathfrak{n}]   \subseteq \mathfrak{p}_{j_0-s-1} + J\mathfrak{p}_{j_0-s-1} \subseteq \mathfrak{d}^s. 
\end{equation*} This implies, using $\eqref{eq:1},$ $  \mathfrak{p}_{j_0-s-2}  +J \mathfrak{p}_{j_0-s-2} \subseteq \mathfrak{d}^{s+1}.$ By induction, $\mathfrak{p}_{j_0-j-1} + J\mathfrak{p}_{j_0-j-1}$   $ \subseteq \mathfrak{d}^j $ for all $j \geq 0.$ In particular, let $j = j_0-1.$ Then $\mathfrak{n} \subseteq \mathfrak{d}^{j_0-1}$, which implies that $J$ is nilpotent of step $j_0-1$ by definition. This is a contradiction. Therefore $ \mathfrak{p}_{j_0-1} \neq \{0\}.$
 
Conversely, suppose that $ \mathfrak{p}_{j_0} = \{0\}$ and $ \mathfrak{p}_{j_0-1} \neq \{0\}$.  We show that $J$ is nilpotent of step $j_0.$ By definition, $  \mathfrak{p}_{j_0} + J \mathfrak{p}_{j_0} = \{0\} = \mathfrak{d}^0 $. It follows, by induction, that $\mathfrak{p}_{j_0-j} + J\mathfrak{p}_{j_0-j} \subseteq \mathfrak{d}^{j} $ for all $j \geq 0.$ Hence $\mathfrak{p}_{j_0-j} \subseteq \mathfrak{d}^j$. In particular, let $j =j_0-1.$ Then $$ \mathfrak{p}_1 =     [\mathfrak{n},\mathfrak{n}] \subseteq  \mathfrak{d}^{j_0-1} \Rightarrow \mathfrak{n}/\mathfrak{d}^{j_0-1} \text{ is Abelian }.$$  By $\text{Lemma }\ref{9},$ $J$ is nilpotent of step at most $j_0.$ 

We next show that $\mathfrak{d}^{j_0-1} \neq \mathfrak{n}$ by contradiction. Assume, by contradiction, that $ \mathfrak{n}= \mathfrak{d}^{j_0-1}.$ We show that $\mathfrak{p}_{j-1} \subseteq \mathfrak{d}^{j_0-j}$ for all $j \geq 1$ by induction. By definition, $\mathfrak{p}_0   = \mathfrak{n}= \mathfrak{d}^{j_0-1}.$  Next, suppose that $\mathfrak{p}_{s-1} \subseteq \mathfrak{d}^{j_0-s}$ for some $s \in \mathbb{N}.$ Then  \begin{align*}
    \mathfrak{p}_s & = [\mathfrak{p}_{s-1},\mathfrak{n}] +[J\mathfrak{p}_{s-1},\mathfrak{n}] \\
    & \subseteq [\mathfrak{d}^{j_0-s},\mathfrak{n}] +[J\mathfrak{d}^{j_0-s},\mathfrak{n}] \\
    & \subseteq \mathfrak{d}^{j_0-s-1}.
\end{align*} By induction, $\mathfrak{p}_{j-1} \subseteq \mathfrak{d}^{j_0-j}$ for all $j \geq 1.$ In particular, let $j = j_0.$ We deduce that $  \mathfrak{p}_{j_0-1} \subseteq \mathfrak{d}^0 = \{0\}.$ This implies that $ \mathfrak{p}_{j_0-1} = \{0\}$ which is a contradiction. Hence $\mathfrak{d}^{j_0-1} \neq \mathfrak{n}.$ By definition, $J$ is nilpotent of step $j_0.$

We now show (i) and (iii) are equivalent. Since $J$ is nilpotent of step $j_0,$ it follows, from $\text{Lemma } \ref{9}$ part (ii), that $\mathfrak{d}_j \subseteq \mathfrak{d}^{j_0 - j}$ for all $j \geq 0.$ In particular, let $j = j_0.$ By definition, $\mathfrak{d}_{j_0} = \mathfrak{d}^0 = \{0\}.$ We show that $\mathfrak{d}_{j_0-1} \neq \{0\}.$ By Lemma $\ref{p}$, $\{0\} \neq   \mathfrak{p}_{j_0-1} + J  \mathfrak{p}_{j_0-1} \subseteq \mathfrak{d}_{j_0-1}.$ Hence $\mathfrak{d}_{j_0-1} \neq \{0\}.$
 
 Conversely, assume that $\mathfrak{d}_{j_0} = \{0\}$ and $\mathfrak{d}_{j_0-1} \neq \{0\}$. By definition, $[\mathfrak{d}_{j_0-1},\mathfrak{n}] \subseteq \mathfrak{d}_{j_0} = \{0\}.$ Hence, $\{0\} \neq \mathfrak{d}_{j_0-1} \subseteq \mathfrak{d}^1.$ Next, assume that $\mathfrak{d}_{j_0-s} \subseteq \mathfrak{d}^s$ for some $s \in \mathbb{N}.$ Then by definition, \begin{align*}
  [\mathfrak{d}_{j_0-s-1},\mathfrak{n}] \subseteq  \mathfrak{d}_{j_0-s} \subseteq \mathfrak{d}^s. 
\end{align*}  By $\eqref{eq:1},$ $\mathfrak{d}_{j_0-s-1} \subseteq \mathfrak{d}^{s+1}.$ By induction, $\mathfrak{d}_{j_0-j} \subseteq \mathfrak{d}^j$ for all $j \geq 0.$ Let $j= j_0.$ We find that $\mathfrak{d}_0 = \mathfrak{n} \subseteq \mathfrak{d}^{j_0}.$ Therefore $\mathfrak{d}^{j_0} = \mathfrak{n}$ and $J$ is nilpotent of step at most $j_0.$ 
 
We next show that $\mathfrak{d}^{j_0-1} \neq \mathfrak{n}.$ Suppose not, that is, $  \mathfrak{n} =\mathfrak{d}^{j_0-1}.$ By definition, $\mathfrak{d}_0 = \mathfrak{n} = \mathfrak{d}^{j_0-1}.$ It follows, by induction, that $\mathfrak{d}_{j-1} \subseteq \mathfrak{d}^{j_0-j}$ for all $j \geq 1$. Let $j = j_0.$ We find that $ \mathfrak{d}_{j_0-1} \subseteq \{0\}.$ This is a contradiction. Hence $\mathfrak{d}^{j_0-1} \neq \mathfrak{n} $ and $J$ is nilpotent of step $j_0.$

Finally, since (i) is equivalent to both (ii) and (iii), we conclude that (ii) and (iii) are equivalent.
\end{proof}

\begin{remark}
\label{r35}
Suppose that a Lie algebra $\mathfrak{n}$ admits a nilpotent complex structure $J$ of step $j_0.$ Then  \begin{equation}
    \mathfrak{c}_j(\mathfrak{n}) + J\mathfrak{c}_j(\mathfrak{n}) \subseteq \mathfrak{p}_j + J \mathfrak{p}_j \subseteq \mathfrak{d}_{j} \subseteq \mathfrak{d}^{j_0-j} \label{eq:ew}
\end{equation} for all $ j \geq 0$. 
\end{remark}
  
It is shown that, in $\cite[\text{Corollary }7]{MR1665327},$ if $\mathfrak{c}_j(\mathfrak{n})$ is $J$-invariant for all $j \geq 0,$ then $J$ is nilpotent. We will provide a different approach to this.    
 
\begin{corollary}
\label{43.9}
Let $\mathfrak{n}$ be a step $k$ nilpotent Lie algebra with a complex structure $J.$ Suppose that all $\mathfrak{c}_j(\mathfrak{n})$ are $J$-invariant. Then $\mathfrak{p}_j = \mathfrak{c}_j(\mathfrak{n})$ for all $j \geq 0.$ Furthermore, $J$ is nilpotent of step $k$.
  \end{corollary}
\begin{proof}
Since all $\mathfrak{c}_j(\mathfrak{n})$ are $J$-invariant, by definition, $\mathfrak{p}_0 = \mathfrak{n} = \mathfrak{c}_0(\mathfrak{n}).$ We have, by induction, that $ \mathfrak{p}_j = \mathfrak{c}_j(\mathfrak{n}) $ for all $j \geq 0.$ Therefore $\mathfrak{p}_k = \mathfrak{c}_k(\mathfrak{n}) = \{0\}$ and $\mathfrak{p}_{k-1} = \mathfrak{c}_{k-1}(\mathfrak{n}) \neq \{0\}$. By $\text{Theorem } \ref{14},$ $J$ is nilpotent of step $k.$
\end{proof}

\begin{corollary}
\label{uus}
Let $\mathfrak{n}$ be a step $k$ nilpotent Lie algebra with a nilpotent complex structure $J$ of step $k$. Suppose that $\mathfrak{c}_{k-1}(\mathfrak{n}) = \mathfrak{z}$. Then $\mathfrak{z}$ is $J$-invariant.
\end{corollary}
 
\begin{proof}
Since $J$ is nilpotent of step $k,$ by $\eqref{eq:ew}$,   \begin{equation*}
       \mathfrak{z} + J \mathfrak{z} \subseteq \mathfrak{d}_{k-1} \subseteq \mathfrak{d}^1 \subseteq \mathfrak{z}\Rightarrow [ \mathfrak{z} + J \mathfrak{z} ,\mathfrak{n}] = \{0\}. 
\end{equation*}
Hence $J\mathfrak{z} = \mathfrak{z}.$  
\end{proof}

\begin{corollary}
\label{42.5}
Let $\mathfrak{n}$ be a Lie algebra with a nilpotent complex structure $J$ of step $j_0$. Then for all $j \geq 1,$ $\mathfrak{d}_{j_0-j}$ is not contained in $\mathfrak{d}^{j-1}$.
\end{corollary}
 
\begin{proof}
Since $J$ is nilpotent of step $j_0,$ by $\text{Theorem }\ref{14},$ $\mathfrak{d}_{j_0-1} \neq \{0\} = \mathfrak{d}^0$. Hence $\mathfrak{d}_{j_0-1}$ is not contained in $\mathfrak{d}^0 .$  Next, suppose that $\mathfrak{d}_{j_0-s+1}$ is not contained in $\mathfrak{d}^{s-2}$ for some $\mathbb{N} \ni s \geq 2.$ We show that $\mathfrak{d}_{j_0-s}$ is not contained in $\mathfrak{d}^{s-1}.$ Suppose not. That is, $\mathfrak{d}_{j_0-s} \subseteq \mathfrak{d}^{s-1}.$ Then \begin{align*}
    \mathfrak{d}_{j_0-s+1}  & = [\mathfrak{d}_{j_0-s},\mathfrak{n}] + J[\mathfrak{d}_{j_0-s},\mathfrak{n}] \\ 
    & \subseteq [\mathfrak{d}^{s-1},\mathfrak{n}]+J[\mathfrak{d}^{s-1},\mathfrak{n}] \subseteq \mathfrak{d}^{s-2}.
\end{align*}  It follows that $\mathfrak{d}_{j_0-s+1} \subseteq \mathfrak{d}^{s-2}.$ This is a contradiction. Hence $\mathfrak{d}_{j_0-s}$ is not contained in $ \mathfrak{d}^{s-1}.$ By induction,  for all $j \geq 1,$ $\mathfrak{d}_{j_0-j}$ is not contain in $\mathfrak{d}^{j-1}$. 
\end{proof}

We investigate the possible range of $\dim \mathfrak{z}$ for a Lie algebra $\mathfrak{n}$ with a nilpotent complex structure $J.$ 

\begin{proposition}
\label{16}
Let $\mathfrak{n}$ be a non-Abelian Lie algebra of dimension $2n$ with a nilpotent complex structure $J.$ Then $ 2 \leq  \dim \mathfrak{z}  \leq 2n-2.$
\end{proposition}
 
\begin{proof}
Recall that $ \mathfrak{d}^1 = \mathfrak{z} \cap J \mathfrak{z}$, which is the largest $J$-invariant subspace of $\mathfrak{z}$. Since $J$ is nilpotent, it is clear that $ \mathfrak{d}^1 \neq \{0\}.$ Furthermore, since $ \mathfrak{d}^1$ is $J$-invariant, it follows that $  2 \leq  \dim  \mathfrak{d}^1 \leq  \dim \mathfrak{z}.$ Then the lower bound of $\dim \mathfrak{z}$ is $2.$

Next, we show that the upper bound of $\dim\mathfrak{z}$ is $2n-2.$ Since $\mathfrak{n}$ is non-Abelian, it is possible to find $X,Y \in \mathfrak{n}$ such that $0 \neq [X,Y] \in \mathfrak{c}_1(\mathfrak{n}).$ Then $ \mathrm{span}  \{X,Y\}$ is $2$-dimensional and $ \mathrm{span}  \{X,Y\} \cap \mathfrak{z} = \{0\}.$ Hence $\dim \mathfrak{z} \leq 2n -2.$ 

In conclusion, $2 \leq  \dim \mathfrak{z} \leq 2n  - 2.$  
\end{proof}

\begin{remark}
From Proposition $ \ref{16},$ we can further conclude that if $\dim \mathfrak{z} = 1,$ then the complex structure $J$ on $\mathfrak{n}$ is non-nilpotent. In particular, the Lie algebra of $n \times n$ upper triangular matrices does not admit a nilpotent complex structure.
\end{remark} 
 
\section{Stratified Lie algebras with complex structures}
 
In this section we consider a special type of nilpotent Lie algebras: $\textit{stratified Lie}$ $\textit{algebra}$. Recent results on nilpotent Lie algebras with a stratification can be found in $\cite{MR4127910},\cite{MR3521656}, \cite{MR3742567}.$ We start with the definition of stratified Lie algebras. 

\begin{definition}
A nilpotent Lie algebra $\mathfrak{n}$ is said to admit a $\textit{step k stratification}$ if it has a vector space decomposition of the form $\mathfrak{n}_1  \oplus \mathfrak{n}_2 \oplus \ldots \oplus \mathfrak{n}_k,$ where $\mathfrak{n}_k \neq \{0\},$ satisfying the bracket generating property $[\mathfrak{n}_1,\mathfrak{n}_k] = \{0\}$ and \begin{align*}
    [\mathfrak{n}_1 ,\mathfrak{n}_{j-1}] = \mathfrak{n}_j \qquad  \text{ for all } j \in  \{2,\ldots,k\} .
\end{align*}  A Lie algebra $\mathfrak{n}$ that admits a stratification is called a $\textit{stratified Lie algebra}.$  A complex structure $J$ on a stratified Lie algebra $\mathfrak{n}$ is said to be $\textit{strata-preserving}$ if it preserves each layer of the stratification.
\end{definition}

\begin{remark}
\label{rs}
Let $\mathfrak{n}$ be a step $k$ stratified Lie algebra. By induction, \begin{equation}
     \mathfrak{c}_j(\mathfrak{n}) =\bigoplus_{ j+1\leq l \leq k}  \mathfrak{n}_l \qquad  \text{ for all }  j \geq 0.  \label{eq:e}  
\end{equation} 
\end{remark}

\begin{proposition}
\label{abc}
Let $\mathfrak{n}$ be a $2n$-dimensional step $n$ nilpotent Lie algebra for some $n \in \mathbb{N}$. Suppose that $\dim  \mathfrak{c}_j(\mathfrak{n}) = 2n-2j$ for all $1 \leq j \leq n.$ Then $\mathfrak{n}$ does not admit a stratification.
\end{proposition}

\begin{proof}
Assume, by contradiction, that $\mathfrak{n}$ admits a stratification. Since, by $\eqref{eq:e}$, $ \mathfrak{c}_j(\mathfrak{n})=\bigoplus_{j+1 \leq l \leq n}$   $ \mathfrak{n}_l$  and  $\dim\mathfrak{c}_1(\mathfrak{n}) = 2n-2$, $\dim \mathfrak{n}_1 =2.$ Since $\mathfrak{n}$ is a stratified Lie algebra, $\mathfrak{n}_2 = [\mathfrak{n}_1,\mathfrak{n}_1].$ Thus $\dim \mathfrak{n}_2 = 1$ and $\dim\mathfrak{c}_2(\mathfrak{n}) = 2n-3 > 2n-4.$ This is a contradiction.  
\end{proof} 

\begin{proposition}
\label{3.4}
Let $\mathfrak{n}$ be a step $k$ stratified Lie algebra with a complex structure $J$ and $k \geq 2$. Suppose that $\dim \mathfrak{n}_1 =2.$ Then $J$ is not strata-preserving.
\end{proposition}
\begin{proof}
Suppose, by contradiction, that there exists a strata-preserving complex structure $J.$ Then $\dim \mathfrak{n}_j \in 2\mathbb{N}$ for all $j \geq 1.$ However, $\dim \mathfrak{n}_1= 2 $ implies that $\dim \mathfrak{n}_2 =1,$ which contradicts the assumption that $\dim \mathfrak{n}_2 \in 2 \mathbb{N}.$ Hence $\mathfrak{n}$ does not have a strata-preserving complex structure.
\end{proof}

\begin{remark}
Let $\mathfrak{n}$ be a step $3$ stratified Lie algebra with a strata-preserving complex structure. Arguing in a similar way as in Proposition $\ref{3.4}$, we conclude that $\dim \mathfrak{n} \neq 4$ or $6.$
\end{remark}

We show that there always exists a stratification on step $2$ nilpotent Lie algebra with a strata-preserving complex structure $J$. 
 
\begin{theorem}
\label{5}
Let $\mathfrak{n}$ be a step $2$ nilpotent Lie algebra with a complex structure $J.$ Suppose that $\mathfrak{c}_1(\mathfrak{n})$ is $J$-invariant. Then $\mathfrak{n}$ admits a $J$-invariant stratification.
\end{theorem}
 
\begin{proof}
Define a $J$-invariant inner product $\psi$ by \begin{align*}
     \psi (X,Y) =  \phi(X,Y) + \phi(JX ,JY)  , \text{ for all } X,Y \in \mathfrak{n},
\end{align*} where $\phi$ is an inner product on $\mathfrak{n}.$ We show that there exists a stratification on $\mathfrak{n}$ such that $\mathfrak{n}_1$ and $\mathfrak{n}_2$ are $J$-invariant. Define $\mathfrak{n}_2 = [\mathfrak{n},\mathfrak{n}]$ and $ \mathfrak{n}_1   = \mathfrak{n}_2^\perp$, the orthogonal complement of $ \mathfrak{n}_2$ with respect to $\psi$. Then $\mathfrak{n}_2 = \mathfrak{c}_1(\mathfrak{n})$ is $J$-invariant and by definition $\mathfrak{n}= \mathfrak{n}_1 \oplus \mathfrak{n}_2$. Also note that $$ \mathfrak{n}_2 = [\mathfrak{n}_1\oplus \mathfrak{n}_2,\mathfrak{n}_1\oplus\mathfrak{n}_2] = [\mathfrak{n}_1,\mathfrak{n}_1].  $$ This implies that $\mathfrak{n}_1$ generates $\mathfrak{n}.$ Thus $J$ is a complex structure that preserves both $\mathfrak{n}_1$ and $\mathfrak{n}_2$.
\end{proof}

\begin{remark}
(i) Let $\mathfrak{g}$ be an arbitrary Lie algebra. A complex structure $J$ on $\mathfrak{g}$ is called $\textit{bi-invariant}$ if $J[X,Y] = [JX,Y]$ for all $X,Y \in \mathfrak{g}.$ That is, $J \circ \ad = \ad \circ J.$ A complex structure $J$ is called $\textit{Abelian}$ if $[X,Y] = [JX,JY]$ for all $X,Y \in \mathfrak{g}.$ See, e.g.,  $\cite{MR2040168}$, $\cite{MR2070596}$. Notice that $J$ preserves all terms of $\mathfrak{c}_j(\mathfrak{n})$ and $\mathfrak{c}^j(\mathfrak{n})$ if $J$ is bi-invariant, while if $J$ is Abelian, $J$ only preserves all terms of $\mathfrak{c}^j(\mathfrak{n}).$ 

(ii) Suppose that $\mathfrak{n}$ is a step $k$ stratified Lie algebra with a bi-invariant complex structure $J.$ From $\eqref{eq:e}$, $ \mathfrak{c}_j(\mathfrak{n}) = \bigoplus_{j+1 \leq l \leq k} \mathfrak{n}_l$, it is clear that $\dim \mathfrak{n}_j \in 2 \mathbb{N}$ for all $ j \in \{1,\ldots,k\}.$ 
\end{remark}

\begin{proposition}
\label{44}
Let $\mathfrak{n}$ be a step $k$ stratified Lie algebra with a strata-preserving complex structure $J.$ Then $J\mathfrak{c}_j(\mathfrak{n}) = \mathfrak{c}_j(\mathfrak{n})$ for all $j \geq 0 $ and $J$ is nilpotent of step $k$.
\end{proposition}
 
\begin{proof}
We first show that $J\mathfrak{c}_j(\mathfrak{n}) = \mathfrak{c}_j(\mathfrak{n})$ for all $j \geq 0.$ Recall, from $\eqref{eq:1},$ that $\mathfrak{c}_j(\mathfrak{n}) $   $= \bigoplus_{j+1 \leq l \leq k} \mathfrak{n}_l$ and hence $ J\mathfrak{c}_j(\mathfrak{n}) =\mathfrak{c}_j(\mathfrak{n}) $ for all $j \geq 0.$ By $\text{Corollary }\ref{43.9}$, $J$ is nilpotent of step $k.$
 \end{proof}

It is known that every step $2$ nilpotent Lie algebra maybe stratified (see, e.g., $\cite{MR3742567}$). We will provide another proof in Theorem $\ref{47}$, that every complex structure on a step $2$ nilpotent Lie algebra is nilpotent of step $2$ or $3$. See, e.g.,  $ \cite[\text{Theorem } 1.3]{gao2020maximal}$ and $\cite[\text{Proposition }3.3]{MR2533671}.$ In what follows, we denote by $\mathfrak{k} = \mathfrak{n}_2 \cap J\mathfrak{n}_2$ the largest $J$-invariant subspace contained in $\mathfrak{n}_2$ and we also remind the reader that $\mathfrak{d}^1 = \mathfrak{z} \cap J\mathfrak{z} $ is the largest $J$-invariant subspace contained in $\mathfrak{z}.$ 
 
\begin{theorem}
\label{47}
Let $ \mathfrak{n} = \mathfrak{n}_1 \oplus \mathfrak{n}_2$ be a step $2$ nilpotent Lie algebra with a complex structure $J$ and a $J$-invariant inner product $\psi$.

(i) Suppose that $\mathfrak{k} = \{0\}$. Then $ \mathfrak{d}_1$ is Abelian and $J$ is nilpotent of step $2.$
 
(ii) Suppose that $  \{0\} \neq \mathfrak{k} \subset \mathfrak{n}_2$. Then $J$ is nilpotent of step $3.$

(iii) Suppose that $\mathfrak{n}_2 = \mathfrak{k} .$ Then $J$ is strata-preserving and nilpotent of step $2.$

In conclusion, $J$ is nilpotent of either step $2$ or $3$.
 
\end{theorem}

\begin{proof}
We start with parts (i) and (ii) together. Suppose that $J\mathfrak{n}_2 \neq \mathfrak{n}_2.$ Then, $\mathfrak{p}_2 = [J\mathfrak{n}_2,\mathfrak{n}] \subseteq \mathfrak{n}_2.$  For all $Z_2 \in \mathfrak{n}_2,$ $X,JX \in \mathfrak{n},$ by the Newlander--Nirenberg condition, \begin{equation}
    [J\mathfrak{n}_2,\mathfrak{n}] \ni [JZ_2,JX] = J[JZ_2,X] \in J[J\mathfrak{n}_2,\mathfrak{n}]. \label{eq:iu}
\end{equation}
    
 This implies that $\mathfrak{p}_2$ is $J$-invariant in $\mathfrak{n}_2$. We now consider the following two possibilities. 

(i) If $\mathfrak{k} = \{0\},$ then from $\eqref{eq:iu},$ we get that $\mathfrak{p}_2 = \{0\}.$ By Theorem $\ref{14},$ $J$ is nilpotent of step $2.$ 

(ii) If $\{0\} \neq \mathfrak{k} \subset \mathfrak{n}_2,$ since $\{0\} \neq \mathfrak{p}_2 \subseteq \mathfrak{k}$ and $J\mathfrak{p}_2 \subset \mathfrak{n}_2$, then by definition, $\mathfrak{p}_3 = \{0\}.$ By Theorem $\ref{14},$ $J$ is nilpotent of step $3.$

Finally, for part (iii), suppose that $\mathfrak{n}_2 = \mathfrak{k} .$ We find that $J$ preserves $\mathfrak{n}_2.$ By Theorem $\ref{5},$ $J$ is strata-preserving. From $\text{Corollary }\ref{44},$ $J$ is nilpotent of step $2.$

In conclusion, $J$ is either nilpotent of step $2$ or $3.$
\end{proof}

\begin{remark}
\label{r47}
(i) If $J$ is nilpotent of step $3,$ then there does not necessarily exist a $J$-invariant stratification. 

(ii) We recall, from $\cite[\text{Theorem }1.3]{gao2020maximal},$ if $\mathfrak{z}$ is not $J$-invariant, then $J$ is nilpotent of step $3.$ From Theorem $\ref{47},$ we have the following table: 

\begin{table}[h] 
    \centering
 \begin{tabular}{ |c|c|c|c|c } 
  \hline
$J$   & Strata-preserving & Non-strata-preserving \\ 
  \hline
$J\mathfrak{z} =\mathfrak{z}$  & $J$ nilpotent of step $2$ & $J$ nilpotent of step $2$ \\ 
  \hline
 $J\mathfrak{z} \neq \mathfrak{z}$ & $J$ nilpotent of step $2$ & $J$ nilpotent of step $3$ \\ 
  \hline
 \end{tabular}   
\caption{\label{2.1}nilpotency of $J$}
\end{table} 
From $\text{Table } \ref{2.1},$ if $J$ is nilpotent of step $2$, then $J$ is either strata-preserving or center-preserving. More precisely, we conclude that either $\mathfrak{k} =\mathfrak{n}_2 \cap J\mathfrak{n}_2 =\{0\}$ or $J\mathfrak{n}_2 = \mathfrak{n}_2.$ Indeed, if $\mathfrak{n}$ is step $2$ nilpotent Lie algebra with a nilpotent complex structure $J$ of step $2,$ $J$ may not be strata-preserving.
\end{remark} 
 
Notice that an even dimensional nilpotent Lie algebra with $\dim \mathfrak{c}_1(\mathfrak{n}) = 1$ has step $2$. There does not exist a $J$-invariant stratification for dimensional reasons. Suppose that $\dim \mathfrak{c}_1(\mathfrak{n}) \geq  2.$ We have the following theorem.
 
\begin{theorem}
\label{49}
Let $\mathfrak{n}$ be a step $2$ stratified Lie algebra with a complex structure $J$. 

(i) Suppose that $\dim \mathfrak{n}_2 =2$. Then  

(a) $J$ is nilpotent of step $2;$ (b) if $\dim \mathfrak{d}^1 = 2,$ then $J\mathfrak{n}_2 = \mathfrak{n}_2$.

(ii) Suppose that $\dim \mathfrak{n}_2 = 2l$ for some $l \geq 2 \in \mathbb{N}$. Furthermore, assume that $\dim \mathfrak{d}^1 \leq 4l-2$ and $J\mathfrak{n}_2 \neq \mathfrak{n}_2$. Then $J$ is nilpotent of step $3.$
\end{theorem}

\begin{proof} 
By Theorem $\ref{47},$ $J$ is nilpotent of either step $2$ or $3.$ 

Start with part (i). Assume that $\dim \mathfrak{n}_2 =2.$ For part (a), notice that $J$ could be either strata-preserving or not. If $J$ is strata-preserving, by Theorem $ \ref{47}$ part (iii), $J$ is nilpotent of step $2$. Otherwise, $J$ is not strata-preserving. Since $\dim \mathfrak{n}_2 =2 $, it follows that $\mathfrak{k} = \{0\}.$ Then by Theorem $ \ref{47}$ part (i), $J$ is Abelian and hence nilpotent of step $2.$   

Next, for part (b), recall that $\mathfrak{d}^1 = \mathfrak{z} \cap J\mathfrak{z}$ is the largest $J$-invariant subspace of $\mathfrak{z} $. Suppose that $\mathfrak{n}_2$ is not $J$-invariant. Then $\mathfrak{k}  = \{0\}.$ From part (i), $J$ is nilpotent of step $2.$ It follows, from $\text{Theorem }\ref{14},$ that $\mathfrak{d}_2 = \{0\}$ and $\mathfrak{d}_1 \subseteq \mathfrak{d}^1.$ However, $\dim \mathfrak{d}_1 = \dim \mathfrak{n}_2 \oplus J\mathfrak{n}_2 =  4 > \dim\mathfrak{d}^1.$ This is a contradiction. Hence $J\mathfrak{n}_2 = \mathfrak{n}_2$.

We now show part (ii). Notice that $l \neq 1.$ Otherwise $\dim \mathfrak{n}_2 = \dim \mathfrak{d}^1 = 2.$ This implies that $J\mathfrak{n}_2 = \mathfrak{n}_2.$  Suppose, by contradiction, that $J$ is not nilpotent of step $3.$ Hence $J$ is nilpotent of step $2.$ Then from Remark $\ref{r47}$ part (ii),  $\mathfrak{k} = \{0\} $ and by definition, $\mathfrak{d}_1  = \mathfrak{n}_2 \oplus J\mathfrak{n}_2\subseteq \mathfrak{d}^1.$ However, $\dim \mathfrak{d}_1 = 4l >    \dim \mathfrak{d}^1.$ This is a contradiction. Hence $\mathfrak{k} \neq \{0\}.$ By Theorem $\ref{47}$ part (ii), $J$ is nilpotent of step $3.$
\end{proof}

\begin{remark}
\label{33}
We can extend the statement of part (i) into a higher step stratification as follows:

Let $\mathfrak{n}$ be a step $k$ stratified Lie algebra with a nilpotent complex structure $J$ of step $k$. Suppose that $\dim \mathfrak{n}_k = 2$ and $\dim \mathfrak{d}^1 =2 $. Then $J\mathfrak{n}_k = \mathfrak{n}_k.$
\end{remark}

\begin{corollary}
\label{36}
Let $\mathfrak{n} = \mathfrak{n}_1 \oplus \mathfrak{n}_2$ be a step $2$ stratified Lie algebra with a complex structure $J$ such that $\dim \mathfrak{n}_2 = 2.$ Then $J$ is center-preserving or strata-preserving or both. Furthermore, suppose that $2 \leq \dim \mathfrak{z} \leq 3$ or $\dim \mathfrak{z} = 4$ and $J\mathfrak{z} \neq \mathfrak{z}.$ Then there exists a $J$-invariant stratification.
\end{corollary}

\begin{proof}
By $\text{Theorem }\ref{49},$ $J$ is nilpotent of step $2$. Then by $\text{Table }\ref{2.1}$, $J\mathfrak{n}_2 = \mathfrak{n}_2$ or $J\mathfrak{z} = \mathfrak{z}$ or both if $\mathfrak{n}_2 = \mathfrak{z}$.

Furthermore, $\dim \mathfrak{d}^1 = 2$ since $2 \leq \dim \mathfrak{z} \leq 3$ or $\dim\mathfrak{z} = 4$ and $J\mathfrak{z} \neq \mathfrak{z}.$ By part (ii) of $\text{Theorem }\ref{49},$ $J\mathfrak{n}_2 = \mathfrak{n}_2$. Furthermore, by $\text{Theorem }\ref{5},$ there exists a $J$-invariant stratification. 
\end{proof}

Suppose that $\mathfrak{n}$ is a $6$ dimensional step $2$ nilpotent Lie algebra with a complex structure. In $\cite[\text{Table } 1]{MR1899353},$ there is a complete classification of complex structures on these algebras. However, no information is provided on whether or not $J$ preserves the strata.  
 
\begin{corollary}[$\cite{MR2763953,MR1899353}$]
\label{6nil}
Let $\mathfrak{n}$ be a $6$ dimensional step $2$ nilpotent Lie algebra with a complex structure $J $ such that $\dim \mathfrak{c}_1(\mathfrak{n}) =2 .$ Then $\mathfrak{n}$ admits a $J$-invariant stratification. 
\end{corollary}
 
\begin{proof}
By Theorem $\ref{47}$ and Proposition $\ref{16}$, $J$ is nilpotent and $2 \leq \dim \mathfrak{z} \leq 4$. If $\dim \mathfrak{z}  = 4,$ $\dim \mathfrak{c}_1(\mathfrak{n}) = 1$ and $J$ is not strata-preserving due to dimensional reasons. We omit this case. Next, assume that $\dim \mathfrak{z} \leq 3.$ This is a direct consequence of Corollary $\ref{36}.$ 
 \end{proof}

In what follows, we focus on higher step stratified Lie algebras with complex structures. 
 
\begin{proposition}
Let $\mathfrak{n}$ be a step $3$ stratified Lie algebra with a complex structure $J$.  Suppose that $J\mathfrak{n}_3 = \mathfrak{n}_3$. Then $J$ is nilpotent of step $3.$
\end{proposition}
 
 \begin{proof}
By the definition of the descending central series $\mathfrak{p}_j$ in $\eqref{eq:p},$ $\{0\} \neq \mathfrak{p}_2 = \mathfrak{n}_3 + [J\mathfrak{c}_1(\mathfrak{n}),\mathfrak{n}].$ On the one hand, suppose that $[J\mathfrak{c}_1(\mathfrak{n}),\mathfrak{n}] = \{0\}.$ We deduce that $\mathfrak{p}_2 = \mathfrak{n}_3$ and hence $\mathfrak{p}_3 = \{0\}$ by definition. Using Theorem $\ref{14},$ $J$ is nilpotent of step $3.$ On the other hand, suppose that $ [J\mathfrak{c}_1(\mathfrak{n}),\mathfrak{n}] \neq \{0\}.$ Then by the Newlander--Nirenberg condition, for all $U \in \mathfrak{c}_1(\mathfrak{n})$ and $X,JX \in \mathfrak{n}$ \begin{align*}
    0 \neq \underbrace{[JU,JX] - J[JU,X]}_\text{$\in [J\mathfrak{c}_1(\mathfrak{n}),\mathfrak{n}] + J[J\mathfrak{c}_1(\mathfrak{n}),\mathfrak{n}] $} = \underbrace{[U,X]+J[U,JX]}_\text{$\in \mathfrak{n}_3$}.
\end{align*} Hence $[J\mathfrak{c}_1(\mathfrak{n}),\mathfrak{n}] \subseteq \mathfrak{n}_3.$ This implies that $\mathfrak{p}_2 \subseteq \mathfrak{n}_3$ and therefore $J\mathfrak{p}_2 \subseteq \mathfrak{n}_3.$ Then $\mathfrak{p}_3 = [\mathfrak{p}_2,\mathfrak{n}] + [J\mathfrak{p}_2,\mathfrak{n}] = \{0\}.$ Again by Theorem \ref{14}, $J$ is nilpotent of step $3.$
 \end{proof}

\begin{proposition}
Let $\mathfrak{n}$ be a $8$ dimensional step $3$ stratified Lie algebra with a complex structure $J$ such that $2\dim \mathfrak{n}_3 = \dim \mathfrak{c}_1(\mathfrak{n}) = 4$. Suppose that $J\mathfrak{n}_3 \neq \mathfrak{n}_3$ and $\dim \mathfrak{z} \leq 3.$ Then $J$ is nilpotent of step $4.$ Furthermore, $\mathfrak{d}_2 =  \mathfrak{n}_3 \oplus J\mathfrak{n}_3$.
\end{proposition}

\begin{proof} 
Since $\mathfrak{n}_3 \subseteq \mathfrak{z},$ $\dim \mathfrak{z} \geq 2.$ By $\cite[\text{Corollary }3.12]{MR4009385},$ $J$ is nilpotent. Then using Remark $\ref{r7}$ (i), $3 \leq j_0 \leq 4$, where $j_0$ is the nilpotent step of $J.$ Suppose, by contradiction, that $J$ is nilpotent of step $3.$ It follows, from the equation $\eqref{eq:ew}$, that $\mathfrak{n}_3 + J\mathfrak{n}_3 \subseteq \mathfrak{d}_2 \subseteq \mathfrak{d}^1 \subseteq \mathfrak{z}.$ On the one hand, since $\dim \mathfrak{z} \leq  3,$ $\dim \mathfrak{d}^1 =2.$ On the other hand, since $J\mathfrak{n}_3 \neq \mathfrak{n}_3$ and $\dim \mathfrak{n}_3 = 2,$ $\mathfrak{n}_3 \cap J\mathfrak{n}_3 = \{0\}$ and therefore $ \dim \mathfrak{n}_3 \oplus J\mathfrak{n}_3 = 4 > \dim \mathfrak{d}^1.$ This is a contradiction. So $J$ is nilpotent of step $4.$
 
We now show that $\mathfrak{d}_2 =  \mathfrak{n}_3 \oplus J\mathfrak{n}_3$. It is sufficient to show that $\mathfrak{d}_2 \subseteq  \mathfrak{n}_3 \oplus J\mathfrak{n}_3.$ By definition, \begin{align*}
    \mathfrak{d}_2 & = [\mathfrak{d}_1,\mathfrak{n}] +J[\mathfrak{d}_1,\mathfrak{n}] \\
    & = \mathrm{span}\left\{ [T,X] + J[T',X'] : \text{ } \forall \text{ } T,T' \in \mathfrak{d}_1, \text{ } \forall \text{ } X,X' \in \mathfrak{n} \right\}.
\end{align*} For all $ T,T' \in \mathfrak{d}_1,$ we may write $T = U + JV$ and $T' = U' + JV'$ where $U,V,U',V' \in \mathfrak{c}_1(\mathfrak{n}).$ Then \begin{equation}
 0 \neq   [T,X] + J[T',X']=   \underbrace{[U,X]+J[U',X']}_\text{$ \in \mathfrak{n}_3 \oplus J\mathfrak{n}_3 $} +[JV,X]  + J[JV',X']. \label{eq:90}
\end{equation} By the Newlander--Nirenberg condition, \begin{align*}
 0\neq  \underbrace{[JV,X] + J[JV,JX]}_\text{$ \in  [J\mathfrak{c}_1(\mathfrak{n}),\mathfrak{n}] + J[J\mathfrak{c}_1(\mathfrak{n}),\mathfrak{n}]$} = J[V,X]-[V,X] \in \mathfrak{n}_3 \oplus J\mathfrak{n}_3.
\end{align*} Hence $ [JV,X] + J[JV', X'] \in \mathfrak{n}_3 \oplus J\mathfrak{n}_3.$  From $\eqref{eq:90},$ $[T,X] + J[T',X'] \in \mathfrak{n}_3 \oplus J\mathfrak{n}_3.$ Hence $\mathfrak{d}_2 \subseteq  \mathfrak{n}_3 \oplus J\mathfrak{n}_3.$ In conclusion, $\mathfrak{d}_2 =  \mathfrak{n}_3 \oplus J\mathfrak{n}_3.$
\end{proof}

\section*{Acknowledgement}
The results of this paper are contained in the author's Master of Science thesis at the University of New South Wales, prepared under the the supervision of Michael. G. Cowling and Alessandro Ottazzi. I would like to give deep thanks to both of them, for guiding and sharing their views on mathematics. They also provided very useful comments and suggestions on the project.

\bibliographystyle{plain}
\bibliography{bibliography.bib}

\end{document}